\documentclass[11pt,a4paper,reqno]{amsart}
\usepackage{amsthm,amsmath,amsfonts,amssymb,amsxtra,appendix,bookmark,dsfont,latexsym,bm,hyperref,delarray,color,euscript,amsgen,amsbsy,amsopn,amscd,latexsym,mathrsfs}

\usepackage[old]{old-arrows}
\newlength{\hookwidth}
\newlength{\hookheight}
\newtheorem{define}{Definition}[section]
\newtheorem{pro}{Proposition}[section]
\newtheorem{Lem}{Lemma}[section]
\newtheorem{cor}{Corollary}[section]

\newtheorem{remark}{Remark}[section]
\theoremstyle{plain} \newtheorem{thm}{Theorem}[section]

\DeclareMathOperator{\dist}{dist}
\DeclareMathOperator{\diam}{diam}

\let\div\undefined
\DeclareMathOperator{\div}{div}
\DeclareMathOperator{\supp}{supp}
\DeclareMathOperator{\id}{id}

\numberwithin{equation}{section}

\begin{document}

\title[Singular elliptic equations with nonstandard growth]{Multiplicity results for nonhomogeneous elliptic equations with singular nonlinearities}
\author[]{Rakesh Arora}
\address[Rakesh Arora]{Department of Mathematics and Statistics, Masaryk University, Building 08,
Kotl\'{a}\v{r}sk\'{a} 2, Brno, 611 37, Czech Republic}
\email{arora@math.muni.cz, arora.npde@gmail.com}

\date{\today}

\begin{abstract}
This paper is concerned with the study of multiple positive solutions to the following elliptic problem involving a nonhomogeneous operator with nonstandard growth of $p$-$q$ type and singular nonlinearities
\begin{equation*}
    \left\{
         \begin{alignedat}{2} 
             {} - \mathcal{L}_{p,q} u
             & {}= \lambda \frac{f(u)}{u^\gamma}, \ u>0
             && \quad\mbox{ in } \, \Omega, \\
             u & {}= 0
             && \quad\mbox{ on } \partial\Omega,
          \end{alignedat}
     \right.
\end{equation*}
where $\Omega$ is a bounded domain in $\mathbb{R}^N$ with $C^2$ boundary, $N \geq 1$, $\lambda >0$ is a real parameter,
$$\mathcal{L}_{p,q} u := \div(|\nabla u|^{p-2} \nabla u + |\nabla u|^{q-2} \nabla u),$$ $1<p<q< \infty$, $\gamma \in (0,1)$, and $f$ is a continuous nondecreasing map satisfying suitable conditions. 
By constructing two distinctive pairs of strict sub and super solution, and using fixed point theorems by Amann \cite{amann1976}, we prove existence of three positive solutions in the positive cone of $C_\delta(\overline{\Omega})$ and in a certain range of $\lambda$. 
\medskip
	
\noindent\textit{Key words: {singular nonlinearities, nonstandard growth, $p$-$q$ Laplacian, multiplicity results, infinite positone problems, fixed point theorems.}}

\medskip
	
\noindent\textit{2010 Mathematics Subject Classification: 35J20, 35J62, 35J75, 35J92, 35B65.
}

\end{abstract}

\maketitle

\tableofcontents

\section{Introduction}
\subsection{State of the art}
The study of problems involving nonhomogeneous operators of $\mathcal{L}_{p,q}$ type emerges from the work of Zhikov \cite{Zhikov-87, Zhikov-95}  where
the models of strongly anisotropic materials were considered in the context of homogenization. In particular, Zhikov considered the following model of functional in relationship to the Lavrentiev phenomenon
$$\mathcal{I}(u):= \int_{\Omega} |\nabla u|^p + r(x) |\nabla u|^q ~dx, \ 0 \leq r(x) \leq L, \ 1<p<q$$
where the modulating coefficient $r(x)$ dictates the geometry of the composite made up by two materials with hardening exponent $p$ and $q$ respectively. \vspace{0.1cm}\\
According to Marcellini’s terminology \cite{Marcellini-86, Marcellini-91}, the functional $\mathcal{I}$ falls into the following category of so-called functionals with nonstandard growth of $p$-$q$ type
$$u \longmapsto \int_{\Omega} f(x, \nabla u) ~dx$$
where the energy density function $f$ satisfies
$$|t|^p \leq f(x,t) \leq 1 + |t|^q, \quad 1 \leq p \leq q.$$ 
The pioneering contributions of Marcellini \cite{Marcellini-86, Marcellini-91} and a series of remarkable papers by Mingione et al. \cite{colombo-15, Colombo-15} have brought to light the subject of studying such functionals with nonstandard growth, and have subsequently motivated many works involving the equations and systems with a nonhomogeneous operator of $\mathcal{L}_{p,q}$ type. \vspace{0.15cm}\\
In this paper, we study the existence and multiplicity results of the following problem involving the nonhomogeneous operator $\mathcal{L}_{p,q}$ with nonstandard growth of $p$-$q$ type and singular nonlinearities
\begin{equation*}
    (P_\lambda) \left\{
         \begin{alignedat}{2} 
             {} - \mathcal{L}_{p,q} u
             & {}= \lambda \frac{f(u)}{u^\gamma}, \ u>0
             && \quad\mbox{ in } \, \Omega, \\
             u & {}= 0
             && \quad\mbox{ on } \partial\Omega,
          \end{alignedat}
     \right.
\end{equation*}
where $\Omega$ is a bounded domain in $\mathbb{R}^N$ with $ \partial \Omega \in C^2$, $N \geq 1$, $\lambda >0$, $\gamma \in (0,1)$, $$\mathcal{L}_{p,q} (u):= \div(|\nabla u|^{p-2} \nabla u + |\nabla u|^{q-2} \nabla u), \quad 1<p<q< \infty.$$
The differential operator $\mathcal{L}_{p,q}$ is also known as $(p,q)$-Laplacian operator and appears in a variety of physical models such as reaction diffusion systems, elasticity theory, quantum physics and transonic flows. The singular problem $(P_\lambda)$ is also known as infinite positone problem (see \cite{Ko}). For a more comprehensive description of applications, we refer to the work \cite{Ball-76, benci-2000, Cherfils-2005}, a survey article \cite{Marano-2018} and its references. \vspace{0.15cm}\\
The study of elliptic equations with singular nonlinearities started mainly with the seminal work of Crandall, Rabinowitz and Tarter \cite{Crandall-77}. Later on, much attention has been paid to the subject, leading to an abundant literature investigating a large spectrum of issues (see \cite{Crandall-77, DiHernRako, DMO}) and surveys \cite{GR1, HM}. We cite here some related work with no intent to furnish an exhaustive list. For $p=2$, multiplicity results for the equation involving singular nonlinearites have been dealt in \cite{Haitao, HirSacShi} for critical nonlinearity and in \cite{Dhanya} for exponential nonlinearity. In \cite{Dhanya-15}, authors proved existence of three solutions for the problem $(P_\lambda)$, under the assumption of existence of two different pairs of sub and super solutions. For the quasilinear case {\it i.e.} $p \neq 2$, Giacomoni et al. \cite{giacomoni2007} have studied the problem $(P_\lambda)$ when $f(u) \equiv\mbox{Cst}$ and perturbed with subcritical and critical nonlinearities. Using variational methods, they proved the existence of multiple solutions in $C^{1, \alpha}(\overline{\Omega})$ when $\gamma\in (0,1)$, and figured out the boundary behavior of weak solutions by constructing suitable sub and super solutions. In \cite{Ko}, Ko et al. studied the problem $(P_\lambda)$ and for a certain range of $\lambda$, they proved the existence of two solutions via constructing two pairs of sub and super solutions. For more contrasting results with different summability conditions on $f$, we refer to the work \cite{BougGiacHern, Can,giacomoni2020, kumar2020} and their references within. For a further detailed review of elliptic equations involving singular nonlinearities we refer to the monograph \cite{GR} and the overview article   \cite{HernManVega}. \vspace{0.1cm}\\
Turning to the equations involving the nonhomogeneous operators and singular nonlinearities, in particular, the operator $\mathcal{L}_{p,q}$ have been recently investigated in \cite{giacomoni2020, kumar2020} and \cite{papageorgiou2020,Sim_son_2021}. In \cite{papageorgiou2020, kumar2020} the authors studied the following singular problem 
\begin{equation}\label{eq:nonhomo}
- \mathcal{L}_{p,q} u = \frac{g(x)}{u^\gamma}, \ u >0 \ \text { in } \, \Omega,\ u = 0 \ \text{ in }\, \partial \Omega  
\end{equation}  
perturbed with subcritical and critical growth nonlinearities respectively. Precisely, for $g(x)= \lambda$ and $\gamma \in (0,1)$, Papageorgiou et al. \cite{papageorgiou2020} proved bifurcation type theorem via variational methods, and Kumar et al. \cite{kumar2020} proved the existence of atleast two solutions by splitting the Nehari manifold set. Very recently, in \cite{giacomoni2020}, Giacomoni et al. proved Sobolev and H\"older regularity results for the minimal weak solution $u$ of the problem \eqref{eq:nonhomo}. The study of three solutions for the singular problem involving nonhomogeneous operator with unbalanced growth of $p$-$q$ type and singular nonlinearities was completely open till now, even in the case of $p=q$. In this regard our work brings new results. To achieve the goal, we use the fixed point theorem by Amann \cite{amann1976} by constructing two different pairs of strict sub and super solutions. For more results concerning the singular problem for nonlocal and nonlinear operators, we refer to \cite{arora2019, arora-2020} and their reference within. \vspace{0.15cm}\\
The main difficulty here is because of the singular nature of the reaction term near the boundary which, in turn, prevents the operator associated to $(P_\lambda)$ to be monotone. To handle this, we transform our problem to an equivalent problem by absorbing the singularity into the operator. Precisely, we formulate a new problem $(\hat{P}_\lambda)$ (see page \pageref{page}) and prove the operator $\hat{T}$ (see Page \pageref{defini}) associated to it, is completely continuous, increasing and has strict invariant property. Similar type of ideas are used by Dhanya et al. \cite{Dhanya-15} in the local case {\it i.e.} $p=q=2$, and Giacomoni et al. \cite{giacomoni-2019} in the nonlocal case. They have showed the operator $\hat{T}$ is strongly increasing with the help of strong comparison principle in both local and nonlocal cases, respectively. But here due to the nonlinearity and nonhomogeneity of the operator $\mathcal{L}_{p,q}$, their approach cannot be applied to the problem $(\hat{P}_\lambda)$ (most substantially \cite[Theorem 3.6]{Dhanya-15} and \cite[Theorem 4.7]{giacomoni-2019}). \vspace{0.15cm}\\ 
To overcome these issues, we explicitly provide the construction of two distinctive pairs of strict sub and super solutions $(u_0, u^0)$, $(v_0, v^0)$, which by its own nature and a comparison principle in \cite[Proposition 6]{papageorgiou2020} implies the strict invariant property of the map $\hat{T}$ in suitable retracts of $C_\delta(\overline{\Omega})$ (see page \pageref{strict}). Unlike the case of $p$-Laplacian, the non-homogeneous nature of the principal operator $\mathcal{L}_{p,q}$ does not allow us to use the scaling of eigenfunctions of $(-\Delta)_p$ for the construction of sub and super solution. This necessitates to look for different choices of scaling function in the construction process. The non-trivial task of finding scaling function is carried out by analyzing the qualitative properties of the weak solution of the singular problem involving a doubly parametrized nonhomogeneous operator $\mathcal{L}_{p,q}^{\alpha, \beta}$ (see problem $(P_{\alpha, \beta})$). By considering the nonhomogeneous property of principal operator $\mathcal{L}_{p,q}$ and new types of scaling of the weak solution of the problem $(P_{\alpha, \beta})$, with appropriate choices of the parameters $\alpha$ and $\beta$ in  $\mathcal{L}_{p,q}^{\alpha, \beta}$, the existence of first pair of the sub and super solution is proved. The existence of second pair of sub and super solution is derived by studying a related non-singular problem in a ball $B(0,R)$ of radius $R$ inscribed in $\Omega$ and by borrowing some ODE techniques from \cite{Ko} to suitably extending it to $\Omega$. Lastly, the main result of this work is accomplished by using the fixed point theorem by Amann \cite{amann1976}. 
\subsection{Functional spaces and description of main results}
We start by introducing some functional spaces and notations which are used throughout the text. Assume $\Omega$ is a bounded domain in $\mathbb{R}^N$ with $C^2$ boundary, $N \geq 1$ and denote $$\delta(x):= \inf_{y \in \partial \Omega}
|x-y| \ \text{and} \ \Omega_\nu:= \{x \in \Omega: \dist(x, \partial \Omega) < \nu\} \ \text{for some} \ \nu >0.$$ For a given positive function $\rho \in C_0(\overline{\Omega})$,
$$C_\rho(\overline{\Omega}):= \{u \in C_0(\overline{\Omega}): \ \text{there exists a} \ c \geq 0 \ \text{such that} \ |u(x)| \leq c \rho(x) \ \text{for all} \ x \in \Omega\}$$
equipped with the norm $$\|u\|_{C_\rho(\overline{\Omega})} := \left\| \frac{u}{\rho}\right\|_{L^\infty(\Omega)}$$ is a Banach space. Furthermore, $C_\rho(\overline{\Omega})$ is an ordered Banach space (OBS) whose associated positive cone $ C_\rho(\overline{\Omega})^*:= \{u \in C_\rho(\overline{\Omega}): u(x) \geq 0 \ \text{for all} \ x \in \Omega\}$ has nonempty interior and normal (see \cite[Theorem 1.5]{amann1976}). We also define a open convex subset of the positive cone $C_\rho(\overline{\Omega})^*$ as 
$$C_\rho(\overline{\Omega})^+:= \left\{ u \in C_\rho(\overline{\Omega})^*: \inf_{x \in \Omega} \frac{u(x)}{\rho(x)} >0\right\}.$$
For $1<p<q<\infty$ and $\vartheta >0$, set
$$\mathcal{F}(\vartheta):=  \frac{q \vartheta}{2 C(N,q)} \min\left\{1, \left(\frac{q \vartheta}{2 C(N,q)}\right)^{\frac{q-p}{p-1}}\right\} \ \text{and} \ C(N,q):= \left(\frac{(N+q-1)^{N+q-1}}{N^N}\right)^{\frac{1}{q-1}}.$$
\noindent 
We impose the following assumptions on the function $f:$
\begin{itemize}
    \item[$(f_0)$] $f \in C^1([0,\infty))$ such that $f(0)>0$.
    \item[$(f_1)$] $f$ is nondecreasing in $\mathbb{R}^+.$
    \item[$(f_2)$] $\lim_{t \to \infty} \frac{f(t)}{t^{p-1+\gamma}} = 0.$
    \item[$(f_3)$] There exists $\vartheta_1, \vartheta_2>0 $ such that $0 < \vartheta_1 < \min\{ \vartheta_2, \mathcal{F}(\vartheta_2)\}$ and $\frac{f(t)}{t^\gamma}$ is nondecreasing in $(\vartheta_1, \vartheta_2).$ 
\end{itemize}
We define a function $\hat{f}$ on $[0, \infty)$ as 
\begin{equation*}
 \begin{aligned}
 \hat{f}(t)= \left\{
 \begin{array}{ll}
 \lambda \left( \frac{f(t)-f(0)}{t^\gamma}\right) &  \text{ if } \ t \neq 0 , \\
 0
             & \text{ if } \ t=0,\\
 \end{array} 
 \right.
 \end{aligned}
 \end{equation*}
 and assume  
\begin{enumerate}
    \item[$(f_4)$] There exists a constant $\hat{k}$ such that $\hat{f}(t) + \hat{k} t$ is increasing in $[0, \infty).$
\end{enumerate}
\textbf{Example:} The function $f$ defined by $f(t)= \exp(\frac{kt}{k+t})$ for any $t \geq 0$ with $k \gg 1$ satisfy assumptions $(f_0)$-$(f_4).$ \vspace{0.15cm}\\  
The notion of weak solution for the problem $(P_\lambda)$ is understood in the following sense: 
\begin{define}\label{def:weak-dual}
A function $u \in W_0^{1,q}(\Omega)$ is said to be a weak sub solution (or super solution) of the problem $(P_\lambda)$ if 
\begin{enumerate}
    \item for every $K \Subset \Omega$ there exists a constant $c_K >0$ such that $\inf_K u(x) \geq c_K$ and $f(u) \in L^{q'}_{loc}(\Omega).$
    \item for every $\xi \in \mathbb{W}:= \bigcup_{\Omega' \Subset \Omega} W_0^{1,q}(\Omega')$, $\xi \geq 0$ following holds
    \begin{equation}\label{eq:notion}
    \int_{\Omega} |\nabla u|^{p-2} \nabla u \cdot  \nabla \xi + |\nabla u|^{q-2} \nabla u \cdot  \nabla \xi ~dx \leq (or \geq) \lambda \int_{\Omega} \frac{f(u)}{u^\gamma}  \xi ~dx.
\end{equation}

\end{enumerate}
A function $u \in W_0^{1,q}(\Omega)$ which is both weak sub solution and weak super solution of $(P_\lambda)$ is called a weak solution of $(P_\lambda).$
\end{define}
\begin{remark}
If $u(x) \geq c \delta(x)$ for some $c>0$ and $f(u) \in L^{q'}(\Omega)$ then by using Hardy's inequality, the set of test functions in \eqref{eq:notion} can be extended from $\mathbb{W}$ to $W_0^{1,q}(\Omega).$ 
\end{remark}
\noindent  For the singular problem $(P_\lambda)$, first we prove the following existence result under the weaker assumption $(f_2)'$:
\begin{itemize}
    \item[$(f_2)'$] $\lim_{t \to \infty} \frac{f(t)}{t^{q-1+\gamma}} = 0.$
\end{itemize}
\begin{thm}\label{thm:existence}
Let $f$ satisfies $(f_0)$-$(f_1)$, $(f_2)'$ and $(f_4)$. Then, for every $\lambda>0$, the problem $(P_\lambda)$ has a minimal weak solution in $W_0^{1,q}(\Omega)$. Furthermore, $ u \in C^{1,l}(\overline{\Omega}) \cap C_\delta(\overline{\Omega})^+$ for some $l \in (0,1).$ 
\end{thm}
\noindent Secondly, we prove the following multiplicity result for the singular problem $(P_\lambda)$,
\begin{thm}\label{thm:three-solution}
Let $f$ satisfies $(f_0)$-$(f_4)$. Then there exists constants $0< \lambda_* <\lambda^*$ such that for every $\lambda \in [\lambda_*, \lambda^*]$, the problem $(P_\lambda)$ has atleast three solutions $u_i \in  W_0^{1,q}(\Omega)$. Additionally, $ u_i \in C^{1,l}(\overline{\Omega}) \cap C_\delta(\overline{\Omega})^+$ for some $l \in (0,1)$ and $i=1,2,3.$
\end{thm}
Turning to the layout of the paper, in section \ref{construct} we explicitly construct two pairs of sub and supersolution for the singular problem $(P_\lambda).$ In section \ref{results}, we prove our existence (Theorem \ref{thm:existence}) and multiplicity results (Theorem \ref{thm:three-solution}). 

\section{Construction of strict sub and super solutions pairs}\label{construct}
To construct the pairs of strict sub and super solution, first we investigate the following problem involving a doubly parametrized nonhomogeneous operator $- \mathcal{L}^{\alpha, \beta}_{p,q}$ and singular nonlinearities:
\begin{equation*}
    (P_{\alpha, \beta}) \left\{
         \begin{alignedat}{2} 
             {} - \mathcal{L}^{\alpha, \beta}_{p,q} u
             & {}=  \frac{\lambda}{u^\gamma}, \ u>0
             && \quad\mbox{ in } \, \Omega, \\
             u & {}= 0
             && \quad\mbox{ on } \partial\Omega.
          \end{alignedat}
     \right.
\end{equation*}
where $\mathcal{L}^{\alpha, \beta}_{p,q} u:= \div(\alpha |\nabla u|^{p-2} \nabla u + \beta |\nabla u|^{q-2} \nabla u)$ for $\alpha, \beta>0$ and $1<p<q < \infty.$ \vspace{0.1cm}\\
We start by stating the following existence and regularity result for the problem $(P_{\alpha, \beta})$ whose detailed proof is presented in Appendix.
\begin{Lem}\label{lem:pre1}
For every $\alpha, \beta >0$ and $\gamma \in (0,1)$, there exists a unique minimal weak solution $u_{\alpha, \beta} \in W_0^{1,q}(\Omega)$ of the problem $(P_{\alpha, \beta}).$  Furthermore, $u_{\alpha,\beta} \in C_\delta(\overline{\Omega})^+ \cap C^{1, l}(\overline{\Omega})$  for some $l \in (0,1).$
 \end{Lem}
\begin{thm}\label{thm:pair1}
Let $f$ satisfies $(f_0)$-$(f_1)$ and $(f_2)'$. Then, for any $\lambda>0$, there exists a pair of strict weak sub solution and weak super solution $(u_0, u^0)$ of the problem $(P_\lambda).$ In addition, $u_0, u^0 \in C^{1,l}(\overline{\Omega}) \cap C_\delta(\overline{\Omega})^+$ for some $l \in (0,1).$
\end{thm}
\begin{proof}
First, we construct our subsolution $u_0.$ For $\eta >0$,  let $w_\eta \in C_0^1(\overline{\Omega}) \cap C_\delta(\overline{\Omega})^+$ be the unique weak solution of the problem 
\begin{equation}\label{eq:no-singlar}
   \left\{
         \begin{alignedat}{2} 
             {} - \mathcal{L}_{p,q} w_\eta
             & {}= \eta, \ w_\eta>0
             && \quad\mbox{ in } \, \Omega, \\
             w_\eta & {}= 0
             && \quad\mbox{ on } \partial\Omega.
          \end{alignedat}
     \right.
\end{equation}
The existence of the unique solution $w_\eta$ for $\eta \in (0,1)$ can be proved via \cite[Proposition 3.1 ]{papageorgiou2020} such that $w_\eta \in C_\delta(\overline{\Omega})^+$ and $w_\eta \to 0$ in $C_0^1(\overline{\Omega})$ as $\eta \to 0^+.$ 
Since $(f_0)$ holds, we can choose $\eta= \eta(\lambda)$ small enough such that $\eta$ and $w_\eta$ satisfies
$$    \eta \leq \frac{\lambda}{2} \frac{f(w_\eta)}{w_\eta^\gamma} \ \text{in} \ \Omega.$$
Hence, by defining $u_0:= w_\eta$ as first strict weak sub solution, we obtain, $u_0$ satisfies 
$$- \mathcal{L}_{p,q} u_0 \leq \frac{\lambda}{2} \frac{f(u_0)}{u_0^\gamma} < \lambda \frac{f(u_0)}{u_0^\gamma} - \chi_{\#} \quad \text{weakly in} \ \Omega$$
where $0< \chi_{\#} \leq \frac{\lambda}{2} \min_{x \in \Omega} f(w_\eta)(x) w_\eta^{-\gamma}(x)$.\vspace{0.1cm}\\
For the construction of supersolution $u^0$, let $u_{\alpha}:= u_{\alpha,1}$ be the solution of the problem $$- \mathcal{L}_{p,q}^{\alpha,1} u_\alpha = \frac{2}{u_\alpha^\gamma} , u_\alpha >0 \ \text{in} \ \Omega \ \text{and} \ u_\alpha=0 \ \text{on} \ \partial \Omega$$
for some $\alpha>0$. From Remark \ref{rem:imp} we know that $\|u_\alpha\|_{L^\infty(\Omega)} \leq C(\alpha)$ and $\lim_{\alpha \to 0} C(\alpha) < \infty$, and $f$ satisfy $(f_2)'$, so we choose $\alpha= {\alpha_*}^{p-q}$ for $\alpha_* >0$ large enough such that
\begin{equation}\label{cond:small:alpha}
\frac{f(\alpha_* \|u_{\alpha}\|_{L^\infty(\Omega)})}{(\alpha_* \|u_{\alpha}\|_{L^\infty(\Omega)})^{q+\gamma-1}} \leq \frac{1}{\lambda \|u_{\alpha}\|_{L^\infty(\Omega)}^{q+\gamma-1}}.   \end{equation}
Define $u^0:= {\alpha_*} u_{\alpha}.$ Then,
\begin{equation*}
    \begin{split}
       - \mathcal{L}_{p,q} u^0 &= -\div(|\nabla u^0|^{p-2} \nabla u^0 + |\nabla u^0|^{q-2} \nabla u^0) \\
       &= {\alpha_*}^{p-1} (-\div(|\nabla u_{\alpha}|^{p-2} \nabla u_{\alpha})) + {\alpha_*}^{q-1} (-\div(|\nabla u_{\alpha}|^{q-2} \nabla u_{\alpha})) \\
       & = {\alpha_*}^{q-1} \left( \alpha_*^{p-q} (-\div(|\nabla u_{\alpha}|^{p-2} \nabla u_{\alpha})) + (-\div(|\nabla u_{\alpha}|^{q-2} \nabla u_{\alpha})) \right) \\
       & \geq  {\alpha_*}^{q-1} (- \mathcal{L}^{\alpha,1}_{p,q} u_{\alpha}) \geq \frac{2{\alpha_*}^{q-1}}{ u_{\alpha}^\gamma} =  \frac{2{\alpha_*}^{q-1+\gamma}}{(u^0)^\gamma}.
    \end{split}
\end{equation*}
Since $(f1)$-$(f2)'$ holds and $\alpha_*$ satisfies \eqref{cond:small:alpha}, therefore we get
\begin{equation*}
  \mathcal{L}_{p,q} u^0 \geq  \frac{
  {\alpha_*}^{q-1+\gamma}}{(u^0)^\gamma} + \chi^{\#} \geq  \frac{\lambda f^*({\alpha_*} \|u_{\alpha}\|_{L^\infty(\Omega)})}{(u^0)^\gamma} + \chi^{\#} \geq \lambda \frac{f(u^0)}{(u^0)^\gamma} + \chi^{\#} \ \text{in} \ \Omega.  
\end{equation*}
where $0< \chi^{\#} \leq \alpha_*^{q-1} \min_{x \in \Omega} u_{\alpha}^{-\gamma}(x).$
\end{proof}
\noindent Denote $R$ be the radius of the ball $B(0,R)$ inscribed in $\Omega$ with $R \leq 1 + \frac{N}{q-1}$ and let $\vartheta^* \in (0, \vartheta_1]$ such that $\overline{f}(\vartheta^*)= \min_{0< t \leq \vartheta^*} \frac{f(t)}{2t^\gamma}$ and define $h \in C([0,\infty)$ given by
\begin{equation*}
 \begin{aligned}
 h(t)= \left\{
 \begin{array}{ll}
  \overline{f}(\vartheta^*) &  \text{ if } \ t \leq \vartheta^* , \\
 \frac{f(t)}{2 t^\gamma}
             & \text{ if } \ t \geq \vartheta_1,\\
 \end{array} 
 \right.
 \end{aligned}
 \end{equation*}
 so that the function $h$ is nondecreasing on $[0, \vartheta_1]$ and $h(t) \leq \frac{f(t)}{2t^\gamma}$ for $t \geq 0.$ Set $\epsilon:= \frac{NR}{N+q-1 }.$
\begin{thm}
There exists $\lambda_*, \lambda^*>0$ and a function $\phi \in C^1(\overline{B(0,R)})$ such that for every $\lambda \in [\lambda_*, \lambda^*]$ the function $\phi$ satisfies
\begin{equation}\label{con}
 \left\{
         \begin{alignedat}{2} 
             {} \ - \mathcal{L}_{p,q} \phi(x)
             & {}\leq  \lambda h(\phi), \ \phi>0
             && \quad\mbox{ in } \, B(0,R), \\
             \phi& {}= 0
             && \quad\mbox{ on } \partial B(0,R),
          \end{alignedat}
     \right. \quad  \text{and} \  \|\phi\|_{L^\infty(B(0,R))} \in [\vartheta_1, \vartheta_2].    
\end{equation}
\end{thm}
\begin{proof}
 For $\chi, \kappa >1$, define $\Upsilon: [0,R] \to [0,1]$ by 
\begin{equation*}
 \begin{aligned}
 \Upsilon(r)=  \left\{
 \begin{array}{ll}
1 &  \text{ if } \ r \in [0, \epsilon], \\
1- \left(1-\left(\frac{R-r}{R-\epsilon}\right)^\kappa\right)^\chi
             & \text{ if } \ r \in (\epsilon, R].\\
 \end{array} 
 \right.
 \end{aligned}
 \end{equation*}
 Let $\vartheta \in (\vartheta_1, \min\{\vartheta_2, \mathcal{F}(\vartheta_2)\})$ (see $(f_3)$) and define $v(r)= \vartheta \Upsilon(r)$ such that $|v'(r)| \leq \vartheta \frac{\chi \kappa}{R-\epsilon}$. Now we prove that there exists a radially symmetric solution $\Phi \in C^1(B(0,R))$ of the following problem:
  \begin{equation*}
    (P_{sym})\left\{
         \begin{alignedat}{2} 
             {} \ - \mathcal{L}_{p,q} \phi(x)
             & {}= \lambda h(v(|x|)), \ \phi>0
             && \quad\mbox{ in } \, B(0,R), \\
             \phi& {}= 0
             && \quad\mbox{ on } \partial B(0,R).
          \end{alignedat}
     \right.
\end{equation*}
From elementary calculations, it is easy to see that $\phi(|x|)$ is radially symmetric solution of $(P_{sym})$ iff $\Phi(r):= \phi(|x|)$ for $|x|=r$ is the solution of the following equivalent problem:  
  \begin{equation*}
    (P'_{sym})\left\{
         \begin{alignedat}{2} 
             {} \ -( r^{N-1} \mathbb{L}_{p,q} (\Phi'(r))'
             & {}= \lambda r^{N-1} h(v(r)), \ \Phi>0
             && \quad\mbox{ in } \, r \in [0,R), \\
             \Phi'(0)& {}= 0,\ \ \Phi(R)=0.
          \end{alignedat}
     \right.
\end{equation*}
where $\mathbb{L}_{p,q}(t)= |t|^{p-2} t + |t|^{q-2} t.$ By integrating from $0$ to $r<R$, we get
$$- r^{N-1} \ \mathbb{L}_{p,q} (\Phi'(r)) = \lambda \int_0^r t^{N-1} h(v(t)) ~dt.$$
Using intermediate value theorem, we get the nonhomogeneous function $\mathbb{L}_{p,q}$ is bijective, monotone and continuous, therefore $\mathbb{L}_{p,q}^{-1}$ is well defined and continuous. Hence, we have
\begin{equation}\label{def:deri}
  - \Phi'(r) =  \mathbb{L}_{p,q}^{-1} \left(\frac{\lambda}{r^{N-1}} \int_0^r t^{N-1} h(v(t)) ~dt\right).  
\end{equation}
Now, by again integrating, we get
\begin{equation}\label{solu}
 \Phi(r) = \Phi(R) + \int_{r}^R \mathbb{L}_{p,q}^{-1} \left(\frac{\lambda}{s^{N-1}} \int_0^s t^{N-1} h(v(t)) ~dt\right) ~ds.   
\end{equation}
By taking $\Phi(R)=0$ and using the fact that $$\Phi'(0) = \lim_{r \to 0} \Phi'(r) = - \lim_{r \to 0} \mathbb{L}_{p,q}^{-1} \left(\frac{\lambda}{r^{N-1}}  h(\vartheta) \int_0^r t^{N-1} ~dt\right)= - \lim_{r \to 0} \mathbb{L}_{p,q}^{-1} \left(\frac{\lambda h(\vartheta) }{N} r \right)=0$$ implies $\Phi$ defined in \eqref{solu} is the solution of the problem $(P'_{sym})$ and $\phi(x):= \Phi(|x|)$ is the solution of $(P_{sym}).$ Now, we claim that there exists $\lambda_*, \lambda^* >0$ such that for all $\lambda \in [\lambda_*, \lambda^*]$ the following holds
\begin{equation}\label{claim}
   \phi(r) \geq v(r) \ \text{for all}\ r \in [0,R] \ \text{ and} \ \|\phi\|_{L^\infty(\Omega)} \leq \vartheta_2 
\end{equation}
and which further implies that $\phi$ is the subsolution of the nonsingular problem \eqref{con} since $h$ is nondecreasing function in $[0, \vartheta_2].$ We observe that $\phi(R)= v(R)=0$, so in order to prove our claim \eqref{claim} it is enough to prove that
$$\phi'(r) \leq v'(r) \ \text{for every} \ r \in [0,R].$$
For $r \in [0, \epsilon]$, $v'(r)=0$ and $\phi'(r) \leq 0$, so the claim holds and for $r \in [\epsilon, R]$ we have
\begin{equation*}
    \begin{split}
        - \phi'(r) &=  \mathbb{L}_{p,q}^{-1} \left(\frac{\lambda}{r^{N-1}} \int_0^r t^{N-1} h(v(t)) ~dt\right) \geq \mathbb{L}_{p,q}^{-1} \left(\frac{\lambda}{R^{N-1}} \int_0^\epsilon t^{N-1} h(v(t)) ~dt\right)\\
        &=  \mathbb{L}_{p,q}^{-1} \left(\frac{\lambda h(\vartheta)}{R^{N-1}} \int_0^\epsilon t^{N-1}  ~dt\right) =  \mathbb{L}_{p,q}^{-1} \left(\frac{\lambda h(\vartheta)}{R^{N-1}} \frac{\epsilon^N}{N}\right).
    \end{split}
\end{equation*}
From the definition of the function $v$ and by the choice of $R$ and $\epsilon$, we get $$|v'(r)| \leq \frac{\vartheta \chi \kappa}{R-\epsilon} \ \text{and} \ \frac{2 (\chi \kappa)^{q-1}}{(R-\epsilon)^{q-1}\epsilon^N} \geq \frac{1}{\epsilon^N}\mathbb{L}_{p,q} \left( \frac{\chi \kappa}{R-\epsilon}\right).$$ Hence, if we choose $\chi, \kappa$ close to $1$ and $\lambda$ such that $$\lambda \geq \lambda_*:= \max\{\vartheta^{p-1}, \vartheta^{q-1}\}  \frac{2 R^{N-1} N }{(R-\epsilon)^{q-1}\epsilon^N h(\vartheta)}$$
we get $-\phi'(r) \geq - v'(r)$ and therefore the claim holds. Now, in order to prove the $L^\infty$ bound of the function $\phi$ in \eqref{claim}, we integrate the equation \eqref{def:deri} for any $s \in [0,R].$ Since $\vartheta \in (\vartheta_1, \vartheta_2)$ and the function $h$ in increasing in $[0, \vartheta_2]$, therefore we have 
\begin{equation*}
\begin{split}
     \phi(s) = \int_s^R \mathbb{L}_{p,q}^{-1} \left(\frac{\lambda}{r^{N-1}} \int_0^r t^{N-1} h(v(t)) ~dt\right) ~dr &\leq \int_s^R \mathbb{L}_{p,q}^{-1} \left(\frac{\lambda h(\vartheta)}{r^{N-1}} \int_0^r t^{N-1} ~dt\right) ~dr \\
     &= \int_{s}^R \mathbb{L}_{p,q}^{-1} \left(\frac{\lambda h(\vartheta)}{N} r \right) ~dr.
\end{split}
\end{equation*}
Now, by choosing $\lambda$ such that $$\lambda \leq \lambda^*:= \left(\frac{\vartheta_2 q}{q-1} \right)^{q-1} \frac{N}{h(\vartheta) R^q}$$ and using the fact that $\mathbb{L}_{p,q}^{-1}(z) \leq z^{\frac{1}{q-1}}$ for $z \geq 0$, we get
\begin{equation*}
\begin{split}
    |\phi(s)| &\leq \int_s^R \left|\mathbb{L}_{p,q}^{-1} \left(\frac{\lambda h(\vartheta)}{N} r \right)\right| ~dr \leq \left(\frac{\lambda h(\vartheta)}{N} \right)^{\frac{1}{q-1}} \int_s^R r^{\frac{1}{q-1}} ~dr\\
    & \leq \left(\frac{\lambda h(\vartheta)}{N} \right)^{\frac{1}{q-1}} \int_0^R r^{\frac{1}{q-1}} ~dr \leq \left(\frac{\lambda h(\vartheta)}{N} \right)^{\frac{1}{q-1}} \left(\frac{q-1}{q}\right) R^{\frac{q}{q-1}}  \leq \vartheta_2.
\end{split}
\end{equation*}
\end{proof}
\begin{remark}
For $\epsilon= \frac{NR}{N+q-1}$, the condition $
      \vartheta \leq \min\{\vartheta_2,  \mathcal{F}(\vartheta_2)\}$ implies $\lambda_* < \lambda^*.$ Therefore, the set $(\lambda_*, \lambda^*)$ is nonempty.
\end{remark}
\begin{thm}
Let $f$ satisfies $(f_0)$-$(f_3)$. Then for any $\lambda \in [\lambda_*, \lambda^*]$, there exists a second pair of strict weak sub solution and weak super solution $(v_0, v^0)$ of the problem $(P_\lambda).$ Moreover, $v_0, v^0 \in C^{1,l}(\overline{\Omega}) \cap C_\delta(\overline{\Omega})^+$ for some $l \in (0,1)$ and $$u_0 \leq v_0 < u^0,\ u_0 < v^0 \leq u^0 \ \text{and} \ v_0 \not\leq v^0.$$
\end{thm}
\begin{proof}
First, we construct the supersolution $v^0$ such that $\|v^0\|_{L^\infty(\Omega)} \leq \vartheta_1$ (see $(f_3)$). Let $u_\beta:= u_{1,\beta}$ is the solution of
$$- \mathcal{L}_{p,q}^{1,\beta} u_\beta = \frac{2}{u_\beta^\gamma} , u_\beta >0 \ \text{in} \ \Omega \ \text{and} \ u_\beta=0  \ \text{on} \ \partial \Omega.$$
From Remark \ref{rem:imp}, we know that $\|u_\beta\|_{L^\infty(\Omega)} \leq C(\beta)$ and $\lim_{\beta \to 0} C(\beta) < \infty$, and since $(f_2)$ holds, so we choose $\beta= \beta^*:= m_{\lambda}^{q-p}$ small enough such that 
\begin{equation}\label{eq:cond:lam:alp}
  m_\lambda \|u_{\beta^*}\|_{L^\infty(\Omega)} \leq \vartheta_1 \quad \text{and} \quad m_\lambda^{p-1+ \gamma} \geq \lambda f(m_\lambda C(\beta^*)).  
\end{equation}  
Define $$v^0(x)= m_{\lambda} u_{\beta^*} \quad \text{where} \ m_{\lambda} \ \text{satisfies} \ \eqref{eq:cond:lam:alp}.$$
Then by using $(f_1)$-$(f_2)$, we obtain
\begin{equation*}
    \begin{split}
       - \mathcal{L}_{p,q} v^0 &= -\div(|\nabla v^0|^{p-2} \nabla v^0 + |\nabla v^0|^{q-2} \nabla v^0) \\
       &= m_\lambda^{p-1} (-\div(|\nabla u_{\beta^*}|^{p-2} \nabla u_{\beta^*})) + m_\lambda^{q-1} (-\div(|\nabla u_{\beta^*}|^{q-2} \nabla u_{\beta^*})) \\
       & \geq m_\lambda^{p-1} (- \mathcal{L}^{1, \beta^*}_{q,p} u_{\beta^*}) \geq m_\lambda^{p-1}\frac{2}{u_{\beta^*}^\gamma} \geq \lambda \frac{f(m_\lambda \|u_{\beta^*}\|_{L^\infty(\Omega)})}{(v^0)^\gamma} + \frac{m_\lambda^{p-1}}{u_{\beta^*}^\gamma} \geq \lambda \frac{f(v^0)}{(v^0)^\gamma} + \varepsilon^*
    \end{split}
\end{equation*}
where $0 < \varepsilon^* \leq m_\lambda^{p-1} \min_{x \in \Omega} u_{\beta^*}^{-\gamma}.$\\
Now, to construct the second subsolution $v_0$, we consider the following nonsingular problem in $\Omega$:
 \begin{equation*}
    (P_{NS})\left\{
         \begin{alignedat}{2} 
             {} \ - \mathcal{L}_{p,q} u
             & {}= \lambda h(u), \ u>0
             && \quad\mbox{ in } \, \Omega, \\
             u& {}= 0
             && \quad\mbox{ on } \partial\Omega.
          \end{alignedat}
     \right.
\end{equation*}
Let $R$ be the radius of the ball $B(0,R)$ inscribed in $\Omega$ and the function $\phi$ satisfies \eqref{con}. Herewith, we extend the function $\phi$ in $\mathbb{R}^N \setminus B(0,R)$ and define 
\begin{equation*}
 \begin{aligned}
 \zeta(x)=  \left\{
 \begin{array}{ll}
\phi(x) &  \text{ if } \ x \in B(0,R), \\
0
             & \text{ otherwise}.\\
 \end{array} 
 \right.
 \end{aligned}
 \end{equation*}
 such that $\zeta \in W_0^{1,q}(\Omega) \cap C^1(\Omega)$  and $\zeta$ satisfies $- \mathcal{L}_{p,q} \zeta \leq \lambda h(\zeta)$ in $\Omega$ and $\zeta=0$ on $\partial \Omega.$ In order to obtain  the strictly positive solution of the nonsingular problem $(P_{NS})$, we iterate the subsolution $\zeta$ in the following way. Let $\psi$ be a weak solution of the following problem
 \begin{equation*}
 (P_{p,q})\left\{
         \begin{alignedat}{2} 
             {} \ - \mathcal{L}_{p,q} \psi + \Theta_\lambda \mathbb{L}_{p,q}(\psi)
             & {} =  g(\zeta),              && \quad\mbox{ in } \, \Omega, \\
             \psi& {}= 0
             && \quad\mbox{ on } \partial \Omega
          \end{alignedat}
     \right.    
\end{equation*}
where $g(t)= \lambda h(t)+ \Theta_\lambda \mathbb{L}_{p,q}(t)$ where $\Theta_\lambda$ is chosen in such a way that $g$ is an increasing function for all $t \geq 0.$ The existence of weak solution of the problem $(P_{p,q})$ can be proved by finding the minimizer of the energy functional $\mathcal{J}_\Theta$ defined on $W_0^{1,q}(\Omega)$ as
$$\mathcal{J}_\Theta(\psi) =  \int_{\Omega} \left(\frac{|\nabla \psi|^p}{p} + \frac{|\nabla \psi|^q}{q}\right) ~dx + \Theta_\lambda \int_{\Omega} \left(\frac{|\psi|^p}{p} + \frac{|\psi|^q}{q}\right) ~dx - \int_{\Omega} g(\zeta) \psi ~dx.$$
Since $g(z) \in L^\infty(\Omega)$, and $W_0^{1,q}(\Omega) \hookrightarrow W_0^{1,p}(\Omega) \cap L^{q}(\Omega)$, $\mathcal{J}_\Theta$ is continuous and coercive on $W_0^{1,q}(\Omega)$. Therefore, there exists a global minimizer $\psi \in W_0^{1,q}(\Omega)$ and since $\mathcal{J}_\Theta \in C^1$, $\psi$ is the weak solution of the problem $(P_{p,q})$ in the sense that 
\begin{equation*}
    \begin{split}
        \int_{\Omega} |\nabla \psi|^{p-2} \nabla \psi \cdot  \nabla \xi + |\nabla \psi|^{q-2} \nabla \psi \cdot  \nabla \xi ~dx &+ \Theta_\lambda \int_{\Omega} |\psi|^{p-2} u \xi + |\psi|^{q-2} u \xi ~dx\\
        &= \int_{\Omega} g(z)  \xi ~dx \ \text{for} \ \xi \in W_0^{1,q}(\Omega).
    \end{split}
\end{equation*}
Since $g(\zeta) \geq 0$ and $J(\psi^+) \leq J(\psi)$, $\psi \geq 0.$ By using elliptic regularity theory, we get $\psi \in L^\infty(\Omega)$ (see \cite[Page 286]{lady1968}), $\psi \in C^{1,k}(\overline{\Omega})$ (see \cite[Page 22, Theorem 1.7]{lieberman1991}) and $\psi \in C_\delta(\overline{\Omega})^+$ (see \cite[Page 111, 120]{pucci2007}). Now, by using the monotonicity of the operator $\mathcal{L} + \Theta_\lambda \mathbb{L}$ and monotonicity of the function $h$ we get 
$$\zeta \leq \psi \quad \text{and} \ - \mathcal{L}_{p,q} \psi + \Theta_\lambda \mathbb{L}_{p,q}(\psi) =  g(\zeta) \leq g(\psi) = \lambda h(\psi) + \Theta_\lambda \mathbb{L}_{p,q}(\psi).$$ Since, $h(t) \leq \frac{f(t)}{2t^\gamma}$ for all $t \geq 0$, we obtain $v_0:=\psi$ is the subsolution of the singular problem of $(P_\lambda)$ for all $\lambda \in [\lambda_*, \lambda^*]$ and satisfies
$$- \mathcal{L}_{p,q} v_0 \leq \lambda \frac{f(v_0)}{v_0^\gamma} - \varepsilon_* \ \text{in} \ \Omega$$
where $0< \varepsilon_* \leq \frac{\lambda}{2} \min_{x \in \Omega} f(\psi)(x) \psi^{-\gamma}(x)$.
\end{proof}
\section{Proof of main results}
\subsection{Existence and multiplicity results}\label{results}
\noindent Before proving our main result, we recall some definitions and results from \cite{amann1976}.
\begin{define}
A nonempty subset $E$ of a metric space $X$ is called a retract of $X$ if there exists a continuous map $r:X \to E$ such that $r\big|_{E}= \id\big|_E.$  
\end{define}

\begin{define}
Let $X$ be a nonempty subset of some Banach space and $f$ be a map from $X$ into a second Banach space. Then $f$ is called compact if it is continuous and if $f(X)$ is relatively compact. The map $f$ is called completely continuous if $f$ is continuous and maps bounded subsets of $X$ into compact sets.
\end{define}
\begin{remark}
Every nonempty closed convex subset $E$ of a Banach space $X$ is a retract of $X$ and every compact map is completely continuous, and the two notions coincide if $E$ is bounded.
\end{remark}
\begin{thm}[Lemma 4.1, \cite{amann1976}]\label{thm:fixedpoints}
Let $X$ be a retract of some Banach space and let $f: X \to X$ be a compact map. Suppose that $X_1$ and $X_2$ are disjoint retracts of $X$, and let $U_k$, $k = 1, 2$, be open subsets of $X$ such that $U_k \subset X_k$, $k = 1, 2$. Moreover, suppose that $f(X_k)
 \subset X_k$ and that $f$ has no fixed points on $X_k \setminus  U_k$, $k = 1, 2$. Then $f$ has at least three distinct fixed points $x_0, x_1, x_2$ with $x_k \in X_k$, $k = 1, 2$, and $x_0 \in X \setminus (X_1 \cup X_2).$
\end{thm}
\begin{cor}[Corollary 6.2, \cite{amann1976}]\label{cor:minf-maxf}
Let $X$ be an ordered Banach space and let $[y_1,y_2]$ be an ordered interval in $X.$ Let $f: [y_1, y_2] \to X$ be an increasing compact map such that $f(y_1) \geq y_1$ and $f(y_2) \leq y_2.$ Then $f$ has a minimal fixed point $x_*$ and a maximal fixed point $x^*.$
\end{cor}
Now, we prove our existence and multiplicity results for the problem $(P_\lambda)$ via a critical point theorem in \cite{amann1976}. For this, we begin by introducing an equivalent formulation of our original problem $(P_{\lambda}).$ Since $f \in C^1([0, \infty))$, then $\hat{f}$ can be treated as continuous function on $[0, \infty)$ such that $\hat{f}(0)=0.$ Precisely, by mean value theorem we obtain $\hat{f}(t)= \lambda f'(t_0) t^{1-\gamma}$ for some $t_0 \in (0,t)$. Then the fact $\lim_{t \to 0} |f'(t)| < \infty$ and $\gamma \in (0,1)$ implies $\hat{f}(0)=0.$ Herewith, we formulate our equivalent problem as \label{page}
\begin{equation*}
    (\hat{P}_\lambda) \left\{
         \begin{alignedat}{2} 
             {} - \mathcal{L}_{p,q} u - \lambda \frac{f(0)}{u^\gamma}
             & {}= \hat{f}(u), \ u>0
             && \quad\mbox{ in } \, \Omega, \\
             u & {}= 0
             && \quad\mbox{ on } \partial\Omega.
          \end{alignedat}
     \right.
\end{equation*}
We begin by introducing the notion of weak solution of the problem $(\hat{P}_{\lambda})$ as:
\begin{define}
A function $u \in W_0^{1,q}(\Omega)$ is said to be a weak solution of the problem $(\hat{P}_\lambda)$ if 
\begin{enumerate}
    \item for every $K \Subset \Omega$ there exists a constant $c_K >0$ such that $\inf_K u(x) \geq c_K$ and $\hat{f}(u) \in L^{q'}_{loc}(\Omega).$
    \item for every $\xi \in \mathbb{W}$, $u$ satisfies
    \begin{equation*}
    \int_{\Omega} |\nabla u|^{p-2} \nabla u \cdot  \nabla \xi + |\nabla u|^{q-2} \nabla u \cdot  \nabla \xi ~dx -\lambda \int_{\Omega} \frac{f(0)}{u^\gamma} \xi ~dx
    = \int_{\Omega} \hat{f}(u)  \xi ~dx.
\end{equation*}
\end{enumerate}
\end{define}
\label{defini}
\noindent We extend the function $f$ and $\hat{f}$ in a continuous manner such that $f(t)= f(0)$ and $\hat{f}(t) =\hat{f}(0)$ for all $t \leq 0.$ To use critical point theorem in \cite{amann1976}, we introduce a map $\hat{T}: C_0(\overline{\Omega}) \to C_0^1(\overline{\Omega})$ as  $\hat{T}(u) = w$ if and only if $w$ is the weak solution of
\begin{equation}\label{eq:fixedpointmap}
(\hat{P}_{\lambda,u}) \left\{
         \begin{alignedat}{2} 
             {} - \mathcal{L}_{p,q} w - \lambda \frac{f(0)}{w^\gamma}
             & {}= \hat{f}(u), \ w>0
             && \quad\mbox{ in } \, \Omega, \\
             w & {}= 0
             && \quad\mbox{ on } \partial\Omega.
          \end{alignedat}
     \right.
\end{equation}
Without loss of generality, we can assume $\hat{k}=0$ in $(f_4)$ ({\it i.e.} $\hat{f}$ is increasing on $\mathbb{R}^+$). If not, then instead of $(\hat{P}_{\lambda,u})$, we can study  
\begin{equation*}
    (\tilde{P}_{\lambda,u}) \left\{
         \begin{alignedat}{2} 
             {} - \mathcal{L}_{p,q} w - \lambda \frac{f(0)}{w^\gamma} + \hat{k} w
             & {}= \hat{f}(u), \ w>0
             && \quad\mbox{ in } \, \Omega, \\
             w & {}= 0
             && \quad\mbox{ on } \partial\Omega.
          \end{alignedat}
     \right.
\end{equation*}
and establish the same results for $(\tilde{P}_{\lambda,u})$ by defining the map $\hat{T}(u)=w$ iff $w$ is a weak solution of $(\tilde{P}_{\lambda,u}).$  
\begin{pro}\label{pro:relation}
A function $u \in W_0^{1,q}(\Omega) \cap C_0^1(\overline{\Omega}) \cap C_\delta(\overline{\Omega})^+$ is weak solution of $(P_\lambda)$ iff $u$ is a fixed point of the map $\hat{T}.$
\end{pro}
\begin{proof}
Let  $u \in W_0^{1,q}(\Omega) \cap C_0^1(\overline{\Omega}) \cap C_\delta(\overline{\Omega})^+$ is weak solution of $(P_\lambda)$, then it implies $u$ is the fixed point of the map $\hat{T}$. Conversely, let $u$ is the fixed point of the map $\hat{T}$ {\it i.e.} $\hat{T}(u)=u$ for some $u \in C_0(\overline{\Omega})$ , then $u$ is the weak solution of $(P_\lambda)$ but it remains to prove that $u \in C_0^1(\overline{\Omega})\cap C_\delta(\overline{\Omega})^+.$ Since $u>0$ and $f$ satisfies $(f_0)$-$(f_2)$ then we have 
\begin{equation*}
 - \mathcal{L}_{p,q} u 
             = \lambda \frac{f(u)}{u^\gamma} \leq \lambda \left(\frac{f(u)}{u^\gamma} \chi_{\{|u| \leq K\}} + C(K) u^{p-1}\right) \leq C(\lambda, f, K)  \left(\frac{1}{u^\gamma} + u^{p-1}\right)  \ \text{weakly in} \ \Omega.
\end{equation*}
Using the same arguments in the proof of \cite[Lemma 3.2]{kumar2020}, we obtain $u \in L^\infty(\Omega).$ Now, let $v_1, v_2 \in W_0^{1,q}(\Omega) \cap C_\delta(\overline{\Omega})^+$ be the weak solution of \begin{equation}\label{sub-super}
    - \mathcal{L}_{p,q} v_1=  \frac{M_1}{v_1^\gamma} \quad \text{and} \  - \mathcal{L}_{p,q} v_2=  \frac{M_2}{v_2^\gamma}
\end{equation}
for $M_1 \leq \lambda f(0)$ and $M_2 \geq \lambda f(\|u\|_{L^\infty(\Omega)}).$
The existence of weak solutions $v_1, v_2$ can be proved via \cite[Theorem 1.4]{giacomoni2020}. Since $f(0) \leq f(u) \leq f(\|u\|_{L^\infty(\Omega)})$, then by using weak comparison principle (see \cite[Theorem 1.5]{giacomoni2020}) we get $C_1 \delta(x) \leq v_1(x) \leq u(x) \leq v_2(x) \leq C_2 \delta(x)$ for all $x \in \Omega$. Finally, \cite[Theorem 1.7]{giacomoni2020} gives $u \in C^{1, l}(\overline{\Omega})$ for some $l \in (0,1).$ 
\end{proof}

\begin{pro}\label{pro:T-prop}
The map $\hat{T}: C_0(\overline{\Omega}) \to C_\delta(\overline{\Omega})^+$ is well defined, completely continuous and increasing.
\end{pro}
\begin{proof}
To show the map $\hat{T}$ is well defined, we have to claim that for every $u \in C_0(\overline{\Omega})$, the problem $(\hat{P}_{\lambda, u})$ has a unique solution in $C_\delta(\overline{\Omega})^+.$ Let $u \in C_0(\overline{\Omega})$ and $\hat{v}= \hat{f}(u).$ Then $\hat{v} \in C_0(\overline{\Omega})$ and $v \geq 0$ in $\Omega.$ Let us consider the following energy functional defined on $W_0^{1,q}(\Omega)$
\begin{equation*}
    \mathcal{E}(w)=  \int_{\Omega} \left(\frac{|\nabla w|^p}{p} + \frac{|\nabla w|^q}{q} \right) ~dx - \lambda f(0) \int_{\Omega} (w^+)^{1-\gamma} ~dx - \int_{\Omega} \hat{v} w ~dx.
\end{equation*}
It is easy to see that the functional $\mathcal{E}$ is coercive and weakly lower semi-continuous on $W_0^{1,q}(\Omega).$ Therefore, there exists a global minimizer $w \in W_0^{1,q}(\Omega)$ and owing to $\mathcal{E}(0)=0 > \mathcal{E}(\epsilon w)$ for every $\epsilon$ small enough and $\mathcal{E}(w) \geq \mathcal{E}(w^+)$, we get $w  \not \equiv 0$ and $w \geq 0.$ \\
Now, we will prove that $w$ is infact the weak solution of the problem $(\hat{P}_{\lambda,u}).$ Let $u_\eta \in W_0^{1,q}(\Omega) \cap C_\delta(\overline{\Omega})^+$ be the weak solution of \eqref{eq:no-singlar}. Due to fact that $u_\eta \in C_\delta(\overline{\Omega})^+$,  $u_\eta \to 0$ in $C_0^1(\overline{\Omega})$ as $\eta \to 0^+$ and Hardy's inequality, we obtain that functional $\mathcal{E}$ is G\^{a}teaux differentiable at $u_\eta$ and hence for $\eta$ small enough
\begin{equation*}
    \mathcal{E}'(u_\eta) = -\mathcal{L}_{p,q} u_\eta - \lambda \frac{f(0)}{u_\eta^\gamma} = \eta - \lambda \frac{f(0)}{u_\eta^\gamma} <0.
\end{equation*}
Define $g: (0,1] \to R$ as $g(t)= \mathcal{E}(w+tv)$ where $v= (u_\eta -w)^+.$ Since $w+tv \geq t u_\eta$ for $t \in (0,1]$ and $g$ is strictly convex, we get $g$ is differentiable on $(0,1]$ and $g'$ is nonnegative and nondecreasing. Therefore, if $|\supp(w)| \neq 0$ $$0 \leq g'(1)- g'(t) \leq g'(1) = \mathcal{E}(u_\eta) <0$$
is a contradiction. Thus, $c \delta(x) \leq w(x)$ in $\Omega$ and $\mathcal{E}$ is differentiable at $w$. This further implies that $w$ is the weak solution of $(\hat{P}_{\lambda,u}).$ Now, to prove the upper boundary behavior and $C^1_0(\overline{\Omega})$ regularity of weak solution $w$, we construct a supersolution of the problem $(\hat{P}_{\lambda, u}).$ Define $\overline{w}= M u_\alpha$ where $u_\alpha:= u_{\alpha,1} \in W_0^{1,q}(\Omega) \cap C_\delta(\overline{\Omega})^+$ is the unique solution of the problem $(P_{\alpha, \beta})$ with $\alpha= M^{p-q}$, $\beta=1$ and $M$ is chosen large enough such that
$$M^{q-1} \left( \frac{1}{u_\alpha^\gamma} - \lambda \frac{f(0)}{(M u_\alpha)^\gamma}\right) \geq \hat{v}.$$
Then,
\begin{equation*}
\begin{split}
  - \mathcal{L}_{p,q} \overline{w} - \lambda \frac{f(0)}{\overline{w}^\gamma}& = - M^{q-1} \left(\div( \alpha |\nabla u_\alpha|^{p-2} \nabla u_\alpha) + \div(|\nabla u_\alpha|^{q-2} \nabla u_\alpha) \right)  - \lambda \frac{f(0)}{\overline{w}^\gamma} \\
  &= - M^{q-1} \mathcal{L}_{p,q}^\alpha u_\alpha  - \lambda \frac{f(0)}{(M u_\alpha)^\gamma} = M^{q-1} \left(  \frac{1}{u_\alpha^\gamma} - \lambda \frac{f(0)}{(M u_\alpha)^\gamma} \right) > \hat{v}.
  \end{split}
\end{equation*}
Now, by using the comparison principle (with minor changes in the proof of \cite[Theorem 1.5]{giacomoni2020})), we get $w(x) \leq C\delta(x)$ in $\Omega.$ Finally  \cite[Theorem 1.7]{giacomoni2020} implies $w \in C^{1, l}(\overline{\Omega})$ and $\hat{T}$ maps bounded subsets of $L^\infty(\Omega)$ to bounded subsets of $C^{1,l}(\overline{\Omega})$ for some $l \in (0,1).$ \vspace{0.1cm}\\ 
To prove the continuity of the operator $\hat{T}$, let $u_n$ be a bounded sequence in $C_0(\overline{\Omega})$ such that $u_n \to u$ in $C_0(\overline{\Omega})$ as $n \to \infty.$ Since $\hat{f}$ is a continuous function and $\hat{f}(0)=0$, $\hat{f}(u_n) \to \hat{f(u)}$ in $C_0(\overline{\Omega})$ as $n \to \infty.$ Denote $w_n= \hat{T}(u_n)$ and $w:= \hat{T}(u).$ As we know that $w_n, w \in W_0^{1,q}(\Omega) \cap C^{1,l}(\overline{\Omega}) \cap C_\delta(\overline{\Omega})^+$ for some $l \in (0,1)$, then by taking $(w_n-w)$ as a test function in the weak formulation of problem $(\hat{P}_{\lambda,u})$ and Lemma \ref{lem:mono}, we get: for $q \geq 2$ \vspace{0.1cm}\\ 
\begin{equation*}
    \begin{split}
       C \int_{\Omega} |\nabla (w_n-w)|^q ~dx &\leq \int_{\Omega} \left(|\nabla w_n|^{p-2} \nabla w_n + |\nabla w|^{q-2} \nabla w \right) \nabla (w_n-w) ~dx \\
       & \quad \quad \quad \quad - \lambda f(0) \int_{\Omega} \left( \frac{1}{w_n^\gamma} - \frac{1}{w^\gamma}\right) (w_n-w)~dx\\
       & = \int_{\Omega} \left(\hat{f}(u_n)- \hat{f}(u) \right) (w_n-w) ~dx \\
       &\leq \|\hat{f}(u_n)- \hat{f}(u)\|_{L^q(\Omega)} \|w_n-w\|_{L^q(\Omega)}\\
       & \leq C \|\hat{f}(u_n)- \hat{f}(u)\|_{L^\infty(\Omega)} \|w_n-w\|_{W_0^{1,q}(\Omega)}.
    \end{split}
\end{equation*}
Since, $w_n$ is uniformly bounded in $ C^{1,l}(\overline{\Omega})$ for some $l \in (0,1)$ and by using interpolation inequality \cite[Corollary 1.3]{le2007}) for any $\theta \in (0,1)$, we get
\begin{equation*}
    \begin{split}
        \|w_n-w\|_{C^1(\overline{\Omega})} &\leq C_1 \|w_n-w\|^{1-\theta}_{C^{1,\gamma}(\overline{\Omega})} \|w_n-w\|_{W_0^{1,q}(\Omega)}^\theta\\
        & \leq C_2  \|\hat{f}(u_n)- \hat{f}(u)\|_{L^\infty(\Omega)}^{\frac{\theta}{q-1}} \leq C_3 \|\hat{f}(u_n)- \hat{f}(u)\|_{C_0(\overline{\Omega})}^{\frac{\theta}{q-1}} \to 0\ \text{as} \ n \to \infty
    \end{split}
\end{equation*}
Since, $\||\nabla w_n|+|\nabla w|\|_{L^\infty(\Omega)} \leq C_0$, $C_0$ is independent of $n$, then by again using interpolation inequality and Lemma \ref{lem:mono} for $q \leq 2$, we get
\begin{equation*}
    \begin{split}
       C_0^{-1} \int_{\Omega} |\nabla (w_n-w)|^2 ~dx \leq C \|\hat{f}(u_n)- \hat{f}(u)\|_{L^\infty(\Omega)} \|w_n-w\|_{W_0^{1,2}(\Omega)}
    \end{split}
\end{equation*}
and $\|w_n-w\|_{C^1(\Omega)} \to 0$ as $n \to \infty.$
\noindent Since $C^{1,l}(\overline{\Omega}) \Subset C^1(\overline{\Omega})$ we have $\hat{T}: C_0(\overline{\Omega}) \to C_0^1(\overline{\Omega})$ is completely continuous. Let $u_1, u_2 \in C_\delta(\overline{\Omega})^+$ such that $u_1 \leq u_2$ in $\Omega$ then there exists $s \in (0,\|u_2\|_{L^\infty(\Omega)})$ such that $\hat{f}(t)= \lambda f'(s) t^{1-\gamma}$ and $\hat{f}$ is increasing in $[0,\|u_2\|_{L^\infty(\Omega)}]$ and $\hat{f}(u_1) \leq \hat{f}(u_2).$ Then, by weak comparison principle we get $\hat{T}(u_1) \leq \hat{T}(u_2)$ {\it i.e.} the map $\hat{T}$ is increasing.
\end{proof}

\noindent \textbf{Proof of Theorem \ref{thm:existence}:}
First, we prove the existence of a fixed point of map $\hat{T}$ defined in \eqref{eq:fixedpointmap} {\it i.e} the existence of weak solution of the equivalent problem $(\hat{P}_\lambda).$\\
Let $X= [u_0, v^0]$ where $u_0,u^0$ are first pair of strict sub and super solution constructed in Theorem \ref{thm:pair1}. Using the same arguments as in Proposition \ref{pro:T-prop}, we get that the map $\hat{T}: X \to X$ is increasing, completely continuous and self invariant. By applying Corollary \ref{cor:minf-maxf}, we get the existence of a minimal fixed point $u_\lambda \in C_\delta(\overline{\Omega})^+ \cap C^{1,l}(\overline{\Omega})$ of map $\hat{T}$ in X for some $l \in (0,1)$. Finally, Proposition \ref{pro:relation} implies that the fixed point $u_\lambda$ is a weak solution of the problem $(P_\lambda).$ 
\qed \vspace{0.1cm}\\
\noindent \textbf{Proof of Theorem \ref{thm:three-solution}:} Define $$X= [u_0, u^0], \quad X_1= [u_0, v^0 ], \quad X_2= [v_0, u^0].$$
The sets $X_i$ for each $i=1,2$ and $X$ form retracts of $C_\delta(\overline{\Omega})$, since they are nonempty, closed and convex subsets of Banach space $C_\delta(\overline{\Omega}).$ The functions $u_0, u^0$ and $\hat{T}(u_0), \hat{T}(u^0)$ satisfies the following inequalities
\begin{equation*}
    \begin{split}
       - \mathcal{L}_{p,q} u_0 - \lambda \frac{f(0)}{u_0^\gamma} & \leq \hat{f}(u_0) - a_* < \hat{f}(u_0) = - \mathcal{L}_{p,q} \hat{T}(u_0) - \lambda \frac{f(0)}{(\hat{T}(u_0))^\gamma}\\
       & \leq - \mathcal{L}_{p,q} \hat{T}(u^0) - \lambda \frac{f(0)}{(\hat{T}(u^0))^\gamma} = \hat{f}(u^0) \\
       & < \hat{f}(u^0) + a^* \leq  - \mathcal{L}_{p,q} u^0 - \lambda \frac{f(0)}{(u^0)^\gamma} \quad \text{in} \ \Omega
    \end{split}
\end{equation*}
for $0< \alpha_* \leq \min\{\varepsilon_*, \chi_{\#}\}$ and $a^* \leq \min\{\varepsilon^*, \chi^{\#}\}.$ Since $u_0, u^0$ are ordered sub and super solutions of $(P_\lambda)$, respectively and $\hat{f}$ in increasing on $[0, \infty)$, then by using strong comparison principle (see \cite[Proposition 6]{papageorgiou2020})  we get 
$$ \hat{T}(u_0) - u_0,\ u^0 - \hat{T}(u^0) \in C_\delta(\overline{\Omega})^+  \ \text{and} \ \hat{T}(X) \subset X.$$
\label{strict}
Using the same arguments as above, Proposition \ref{pro:T-prop},  Corollary \ref{cor:minf-maxf}, we get $$\hat{T}(v_0) - v_0,\ v^0 - \hat{T}(v^0) \in C_\delta(\overline{\Omega})^+, \  \hat{T}(X_i) \subset X_i  \ \text{for}\  i=1,2$$ and there exists a minimal fixed point $\hat{T}(u_1) = u_1  \in X_1$ such that $u_1 \in (u_0, v^0)$ and a maximal fixed point $\hat{T}(u_2)= u_2 \in X_2$ such that $u_2 \in (v_0, u^0).$ Since the map $\hat{T}$ is increasing, then there exist positive constant $c_i>0$ for $i=1,2$ and $\Theta \in C_\delta(\overline{\Omega})^+$ such that
$$u_0 + c_1 \Theta \leq \hat{T}(u_0) \leq  \hat{T}(u_1) =u_1,\quad  v^0 - u_1 = v^0 - \hat{T}(u_1) \geq v^0 - \hat{T}(v^0) \geq c_1 \Theta$$
and
$$u^0 + c_2 \Theta \geq \hat{T}(u^0) \geq  \hat{T}(u_2) =u_2,\quad  u_2 - v_0 = \hat{T}(u_2) - v_0 \geq \hat{T}(v_0) - v_0 \geq c_2 \Theta.$$
For fixed points $u_i$, we define the open ball $B_i$ in $X$ by
$$B_i:= X \cap \{\phi \in C_\delta(\overline{\Omega})^+: \|u_i - \phi\|_{C_\delta(\overline{\Omega})} \leq c_i \}.$$
Then for every $i=1,2$, we have $u_i + B_i \subset X_i$ {\it i.e.} $X_i$ have nonempty interior. So, we construct open balls around each fixed points of the map $\hat{T}$ in $X_i$ and by taking union of all such balls $B_i$ say $\mathbf{B}_i$ such that $X_i \setminus \mathbf{B}_i$ contains no fixed map of $\hat{T}.$ Finally, by using Theorem \ref{thm:fixedpoints}, we get the existence of third point $u_3$ in $X \setminus (X_1 \cup X_2).$ 
\qed

\section{Appendix}
In this section, we recall suitable inequalities due to Simon \cite{Simon} and prove some preliminary results:  
\begin{Lem}\label{lem:mono}
For any $u,v \in \mathbb{R}^N$, we have
\begin{equation*}
    \left< |\nabla u|^{q-2} u- |\nabla v|^{q-2} v, u-v \right> \geq c(q)  |\nabla u - \nabla v|^q  \quad \text{if} \quad q \geq 2
\end{equation*}
\begin{equation*}
   \left< |\nabla u|^{q-2} u- |\nabla v|^{q-2} v, u-v \right> \geq c(q)  \frac{|\nabla u - \nabla v|^2}{(|\nabla u| + \nabla v)^{2-q}} \quad \text{if} \quad 1< q < 2.
\end{equation*}
\end{Lem}
\begin{Lem}\label{lem:singularcase}
Let $1 < q<2$. Then there exists a constant $C>0$ such that for any $u_1, u_2 \in W_0^{1,q}(\Omega)$:
\begin{equation}
\begin{split}
    \int_{\Omega} &\left(|\nabla u_2|^{q-2} \nabla u_2 - |\nabla u_1|^{q-2} \nabla u_1\right)(\nabla u_2 - \nabla u_1) ~dx \\
    & \geq C \left( \int_{\Omega} |\nabla (u_2-u_1)|^q ~dx  \left(\|(|\nabla u_2| + |\nabla u_1|)^\Theta\|_{L^{\frac{2}{2-q}}}\right)^{-1} \right)^{\frac{2}{q}}
\end{split}
\end{equation}
where $\Theta= \frac{q(2-q)}{2}.$
\end{Lem}
\begin{proof}
First, using H\"older inequality, we obtain:
\begin{equation}\label{est:singcase1}
    \int_{\Omega} |\nabla (u_2-u_1)|^q ~dx \left(\|(|\nabla u_2| + |\nabla u_1|)^\Theta\|_{L^{\frac{2}{2-q}}}\right)^{-1} \leq C \left\|\frac{|\nabla (u_2-u_1)^q|}{(|\nabla u_2| + |\nabla u_1|)^\Theta}\right\|_{L^\frac{2}{q}(\Omega)}
\end{equation}
On the other hand, from Lemma \ref{lem:mono},
\begin{equation}\label{est:singcase2}
    \int_{\Omega} \left(|\nabla u_2|^{q-2} u_2- |\nabla u_1|^{q-2} u_1\right) \nabla (u_2-u_1) ~dx  \geq c(q) \int_{\Omega} \frac{|\nabla u_2 - \nabla u_1|^2}{(|\nabla u_2| + \nabla u_1)^{2-q}} ~dx
\end{equation}
Now, by combining \eqref{est:singcase1} and \eqref{est:singcase2}, we obtain our claim.
\end{proof}
Before proving the existence and regularity result for the doubly paramterized problem $(P_{\alpha, \beta})$, we recall the following lemma from \cite{giacomoni2007} for the construction of the barrier function for the problem $(P_{\alpha,\beta}):$
\begin{Lem}[Lemma A.7, \cite{giacomoni2007}] There exists a $C^1$ function $\Xi_\tau: [0, R_\tau] \to [0, \infty]$ satisfying the initial value problem 
\begin{equation}\label{eq:bigest}
    \left\{
         \begin{alignedat}{2} 
             {} -\frac{d}{dr} (|\Xi_\tau'(r)|^{q-2} \Xi_\tau'(r))
             & {}= \frac{1}{\Xi_\tau^\gamma(r)}, \ w_\eta>0
             && \quad\mbox{ for} \ \, r \in (0, R_\tau), \\
             \Xi_\tau(0) & {}= 0, \quad \Xi_\tau'(0)= \tau >0.
          \end{alignedat}
     \right.
\end{equation}
where $R_\tau = \sup \{ s \in (0, \infty): \Xi_\tau'(t) >0 \ \text{for all} \ t \in (0,s)\}.$ The function $\Xi_\tau$ is represented in the form 
$$\Xi_\tau(r)= \tau^\frac{q}{\gamma-1} \Xi_1 (\tau^{\frac{-q}{q+\gamma-1}} r)\ \text{for} \ 0 \leq r \leq R_\tau:= \tau^{\frac{-q}{q+\gamma-1}} R_1 > \diam(\Omega).$$
where the function $\Xi_1$ and $R_1$ are given by \cite[(A.14)]{giacomoni2007} and \cite[(A.15)]{giacomoni2007} respectively. Furthermore, $\Xi_\tau$ is strictly increasing in $[0,R_\tau]$ and $\Xi_\tau'$ is strictly decreasing in $[0,R_\tau)$ and
\begin{equation}\label{eq:smallest}
  -\frac{d}{dr} (|\Xi_\tau'(r)|^{p-2} \Xi_\tau'(r)) \geq 0 \ \text{for} \ r \in (0, R_\tau).  
\end{equation}
\end{Lem}
\noindent \textbf{Proof of Lemma \ref{thm:pair1}}:\vspace{0.1cm}\\
The existence of solution $u_{\alpha, \beta}$ in $W_0^{1,q}(\Omega) \cap C_\delta(\overline{\Omega}) \cap C^{1,l}(\overline{\Omega})$ can be proved by adopting same arguments from \cite[Lemma 2.6 and Proposition 2.7]{giacomoni2020} and \cite[Lemma 3.2]{kumar2020}. For the sake of completeness, we provide a brief sketch of the proof. To prove the existence of a solution, we define the energy functional $\mathcal{J}: W_0^{1,q}(\Omega) \to \mathbb{R}$ as
$$\mathcal{J}(u):= \frac{\alpha}{p}\int_{\Omega} |\nabla u|^p ~dx + \frac{\beta}{q}\int_{\Omega} |\nabla u|^q ~dx  - \frac{\lambda}{1-\gamma} \int_{\Omega} |u|^{1-\gamma} ~dx.$$
Using Sobolev's embedding and lower semi-continuity of the norms, we infer that $\mathcal{J}$ is coercive and weakly lower semi-continuous. Therefore there exists a global minimizer $u \in W_0^{1,q}(\Omega)$ and w.l.g. we can take $u \geq 0.$ Using the strict convexity of the energy functional $\mathcal{J}$ on the positive cone $W_0^{1,q}(\Omega)^+$ and \cite[Lemma 3.1]{giacomoni2007}, we obtain $u(x) \geq c \delta(x)$ in $\Omega$ for some $c>0$ and $\mathcal{J}$ is G\^ateaux differentiable at $u$ (for more details see \cite[Lemma 2.6]{giacomoni2020}), hence $u$ is weak solution of $(P_\lambda).$ By following the same lines of the proof of \cite[Theorem 2]{he2008} and \cite[Lemma 3.2]{kumar2020}, we get $\|u_{\alpha, \beta}\|_{L^\infty(\Omega)} \leq C(\alpha, \beta).$ \vspace{0.1cm}\\
To prove the upper boundary behavior, we define the barrier function $$w(x):= M_\lambda \Xi_\tau(\delta(x)) \ \text{ for} \ x \in \Omega$$ where the choice of $M_\lambda:= M(\lambda) \geq 1$ will be determined later. Since $\partial \Omega \in C^2$, from \cite[Lemma 14.6]{gilbarg} there exists a constant $\nu >0$ such that $\delta \in C^2(\Omega_\nu)$ and by using the fact that $|\nabla \delta|=1$, we get for $ 0 \leq \xi \in C_c^\infty(\Omega_\nu)$
\begin{equation*}
\begin{split}
  \int_{\Omega_\nu} - \mathcal{L}^{\alpha,\beta}_{p,q} w \xi ~dx &=  \beta M_\lambda^{q-1} \int_{\Omega_\nu} (\Xi_\tau' (\delta))^{q-1} \nabla \delta \cdot \nabla \xi ~dx + \alpha M_\lambda^{p-1} \int_{\Omega_\nu} (\Xi_\tau' (\delta))^{p-1} \nabla \delta \cdot \nabla \xi ~dx\\
  &= - \beta M_\lambda^{q-1} \int_{\Omega_\nu} \nabla \left((\Xi_\tau' (\delta))^{q-1} \nabla \delta \right) \xi ~dx  - \alpha M_\lambda^{p-1} \int_{\Omega_\nu} \nabla \left((\Xi_\tau' (\delta))^{p-1} \nabla \delta \right) \xi ~dx\\
  &\geq - \beta M_\lambda^{q-1} \int_{\Omega_\nu} \bigg(( (\Xi_\tau'(\delta))^{q-1})' + \|\Delta \delta\|_{L^\infty(\Omega_\nu)} (\Xi_\tau'(\delta))^{q-1}\bigg) \xi ~dx\\
  & \quad \quad - \alpha M_\lambda^{p-1} \int_{\Omega_\nu} \bigg(( (\Xi_\tau'(\delta))^{p-1})' + \|\Delta \delta\|_{L^\infty(\Omega_\nu)} (\Xi_\tau'(\delta))^{p-1}\bigg) \xi ~dx.
 \end{split}
\end{equation*}
Now, by combining \eqref{eq:bigest} and \eqref{eq:smallest} together with above estimates, we get
\begin{equation*}
    \begin{split}
     - \mathcal{L}^{\alpha,\beta}_{p,q} w &\geq   \beta M_\lambda^{q-1}  \bigg((1-q)  (\Xi_\tau'(\delta))^{q-2} \Xi_\tau^{''}(\delta) - \|\Delta \delta\|_{L^\infty(\Omega_\nu)} (\Xi_\tau'(\delta))^{q-1}\bigg)  \\
     &\quad \quad  \quad \quad +  \alpha M_\lambda^{p-1}  \bigg((1-p)  (\Xi_\tau'(\delta))^{p-2} \Xi_\tau^{''}(\delta) - \|\Delta \delta\|_{L^\infty(\Omega_\nu)} (\Xi_\tau'(\delta))^{p-1}\bigg) \\
    & \geq M_\lambda^{q-1}  \bigg[\beta (1-q)  (\Xi_\tau'(\delta))^{q-2} \Xi_\tau^{''}(\delta) - \|\Delta \delta\|_{L^\infty(\Omega_\nu)} \left(\beta (\Xi_\tau'(\delta))^{q-1} + \alpha (\Xi_\tau'(\delta))^{p-1} \right)\bigg]\\
    & \geq M_\lambda^{q-1}  \bigg[ \frac{\beta}{\Xi_\tau^\gamma(\delta)} - \|\Delta \delta\|_{L^\infty(\Omega_\nu)} \left( \beta (\Xi_\tau'(\delta))^{q-1} + \alpha (\Xi_\tau'(\delta))^{p-1} \right)\bigg]
      \end{split}
\end{equation*}
weakly in $\Omega_\nu.$ Since $\Xi_\tau(0)= 0$, $\Xi_\tau'(0)= \tau$, the concavity of the function $\Xi_\tau$ implies $$\Xi_\tau (\delta) \leq \tau \delta \ \ \text{and} \ \  \Xi_\tau' (\delta)  \leq \tau.$$
Now by choosing $\nu$ small enough such that $\frac{1}{\Xi_\tau^\gamma(\nu)} \geq \frac{1}{2} \left(\|\Delta \delta\|_{L^\infty(\Omega_\nu)} \tau^{p-1} (\beta \tau^{q-p} + \alpha) \right) \ \text{in} \ \Omega_\nu$ we get 
$$- \mathcal{L}^{\alpha,\beta}_{p,q} w \geq  \frac{M_\lambda^{q-1}}{2} \frac{\beta}{\Xi_\tau^\gamma(\delta)} \geq \frac{M_\lambda^{q-1+\gamma}}{2} \frac{\beta}{w^\gamma}  \ \text{ weakly in}\ \Omega_\nu.$$
Now by choosing $M_\lambda$ large enough such that $M_\lambda^{q-1+\gamma} \beta \geq \lambda$ and using the fact that $u_{\alpha, \beta} \in L^\infty_{loc}(\Omega)$, weak comparison principle gives $u_{\alpha, \beta} \leq C \delta$ in $\Omega$ for some $C>0.$ Finally, \cite[Theorem 1,7]{giacomoni2020} implies $u_{\alpha, \beta} \in C^{1,l}(\overline{\Omega})$ for some $l \in (0,1).$
\qed
\begin{remark}\label{rem:imp}
Let $u_{\alpha, \beta}$ be weak solution solution of the problem $(P_{\alpha, \beta})$ such that $$\|u_{\alpha, \beta}\|_{L^\infty(\Omega)} \leq C_1(\alpha, \beta) \ \text{and} \ \|u_{\alpha, \beta}\|_{C^{1,\ell}(\overline{\Omega})} \leq C_2(\alpha, \beta).$$ Then, by looking discreetly the proof of \cite[Theorem 2]{he2008}, \cite[Lemma 3.2]{kumar2020} and using \cite[Remark 1.10]{giacomoni2020}, we get $$\lim_{\alpha \to 0} C_1(\alpha, \beta) \leq C(\beta) \ \text{and} \ \lim_{\alpha \to 0} C_2(\alpha, \beta) < C(\beta).$$
\end{remark}

\section*{Acknowledgement}
The author acknowledge the support from Czech Science Foundation, project GJ19-14413Y and would like to thank Prof. Jacques Giacomoni who provided valuable comments
on the earlier version of the manuscript.

\bibliographystyle{siam}

\end{document}